\documentclass[10pt,reqno]{amsart}
\usepackage{amsmath,amssymb}
\usepackage[margin=2.5cm]{geometry}

\usepackage{amsmath,amssymb,amsfonts,mathrsfs,verbatim,enumitem,pstricks,amsthm, graphicx}
\usepackage[all]{xy}
\usepackage{centernot, xr, accents}
\usepackage{hyperref,url}

\usepackage{amsmath,amssymb,mathtools}
\usepackage[all]{xy}
\usepackage{hyperref}
\usepackage{verbatim}

\usepackage[bottom]{footmisc}
\usepackage{lipsum,cleveref}

\usepackage{hyperref}
\hypersetup{
	colorlinks,
	citecolor=red,
	filecolor=black,
	linkcolor=blue,
	urlcolor=black
}

\theoremstyle{plain}
\newtheorem{thm}{Theorem}
\newtheorem{con}[thm]{Conjecture}
\newtheorem{cor}[thm]{Corollary}
\newtheorem{lmm}[thm]{Lemma}
\newtheorem{prp}[thm]{Proposition}

\newtheorem*{theorm}{Theorem}

\theoremstyle{definition}

\newtheorem*{rem}{Remark}

\theoremstyle{definition}

\numberwithin{equation}{section}

\newcommand{\C}{\mathbb{C}}

\newcommand{\cO}{\mathcal{O}}
\newcommand{\cH}{\mathcal{H}}

\newenvironment{myproof}[1] {\textit{Proof of {#1} :}}{\hfill$\square$}

\def \hf{\hspace*{0.5cm}}

\begin{document}
\title{On endomorphism of algebraic varieties}
%\author[C. Chaudhuri]{Chitrabhanu Chaudhuri}
\author[N. Das]{Nilkantha Das}
%\address{School of Mathematics, IISER Pune}
%\email{chitrabhanu@iiserpune.ac.in}
\address{School of Mathematical Sciences, National Institute of Science 
         Education and Research, Bhubaneswar, HBNI, Odisha 752050, India}
\email{nilkantha.das@niser.ac.in }

\begin{abstract}
We prove that a quasi-finite endomorphism of an algebraic variety over an algebraically closed field of characteristic zero, that is injective on the complement of a closed subvariety, is an automorphism. We also prove that an endomorphism of complex algebraic variety that is injective on the complement of a closed subvariety of codimension at least $2$, is an automorphism. 
\end{abstract}

\maketitle
\tableofcontents

\section{Introduction}
\hf Let $k$ be an algebraically closed field of characteristic zero. Under what conditions an endomorphism of an algebraic variety over $k$ will be an automorphism is an active area of research. For example, an \'{e}tale endomorphism of an affine space is an automorphism is an open question which is known as so-called the \textit{Jacobian conjecture}. The following remarkable result along this line is due to Ax \cite{Ax}:
\begin{theorm}[Ax]
Let $X$ be an algebraic variety over an algebraically closed field of characteristic zero and $\phi:X \rightarrow X$ be an endomorphism of $X$. The endomorphism $\phi$ is injective if and only if it is an automorphism.
\end{theorm} 
This theorem has been proved later in several different approaches, for example \cite{Boel}, \cite{litaka}. In $2005$, M. Miyanishi \cite{Miyanishi} proposed the following conjecture:
\begin{con}\label{main_conj}
Let $\phi:X \longrightarrow X$ be an endomorphism of an algebraic variety $X$ over an algebraically closed field of characteristic zero and, let $Y$ be a proper closed subvariety of $X$ such that the restriction of $\phi$ to $X \setminus Y$ is injective. Suppose furthermore that either $\phi$ is quasi-finite or $Y$ has codimension at least $2$ in $X$. Then $\phi$ is an automorphism.
\end{con}    
\hf This conjecture can be thought as a generalization of the theorem of Ax. We refer to this conjecture as Miyanishi conjecture. In the case of  $\text{codim}_XY \geq 2$, if $X$ is either affine or complete, an affirmative answer to this conjecture is known by Kaliman \cite{Kaliman} (`$\text{codim}_XY$' means the codimension of $Y$ in $X$). Also, an example is given there to assert the fact that the conjecture in the case of $\text{codim}_XY \geq 2$, can't be extended to schemes unlike the theorem of Ax. In this article, we will show that the conjecture is true in general. More precisely, we prove the following theorems:
\begin{thm}\label{main_thm_quasi_finite_case}
Let $\phi:X \longrightarrow X$ be an endomorphism of an algebraic variety $X$ over an algebraically closed field $k$ of characteristic zero and, let $Y$ be a proper closed subvariety of $X$ such that the restriction of $\phi$ to $X \setminus Y$ is injective. If $\phi$ is quasi-finite, then it is an automorphism.
\end{thm}  
 In the case of  $\text{codim}_XY \geq 2$, we consider the case when the field is $ \C$, the field of complex numbers.
\begin{thm}\label{main_thm_codim_case}
Let $\phi:X \longrightarrow X$ be an endomorphism of a complex algebraic variety $X$ and, let $Y$ be a proper closed subvariety of $X$ such that the restriction of $\phi$ to $X \setminus Y$ is injective. If $Y$ has codimension at least $2$ in $X$, then $\phi$ is an automorphism.
\end{thm}
\hf In \Cref{main_thm_codim_case} we restrict ourselves to the case when the underlying field is $\C$. The main reason is that we use the tools from theory of analytic spaces to prove it. \\
%It is unknown to the author whether `Lefschetz principle' can be applied to conclude the result for any algebraically closed field of characteristic zero.\\ 
%\hf The following corollary is an easy application of `Lefschetz principle': 
%\begin{cor}
%Let $\phi:X \longrightarrow X$ be an endomorphism of an algebraic variety $X$ over an algebraically closed field $k$ of characteristic zero and,  let $Y$ be a proper closed subvariety of $X$ such that the restriction of $\phi$ to $X \setminus Y$ is injective. If $Y$ has codimension at least $2$ in $X$, then $\phi$ is an automorphism.
%\end{cor}
\textbf{Notation:} The field $k$ is always assumed to be algebraically closed field of characteristic zero, unless otherwise specified. By algebraic variety over $k$, we mean an integral separated scheme of finite type over $k$ which we denote by $\left( X, \cO_X\right)$. In the case of $k=\C$, we need to consider the corresponding analytic space to $\left( X, \cO_X\right)$ which we denote by $\left(X^{\text{an}},\cH_X\right)$.
The open sets in Zariski topology and analytic topology will be denoted by Z-open and open respectively in order to distinguish the topologies as well. For any subset $A$ of an algebraic variety $X$, the closure of $A$ in the analytic topology will be denoted by $\overline{A}$. For a sheaf of ring $\mathcal{F}$ on $X$, we denote the associated ring over an open subset $U$ of $X$ by $\Gamma\left( U,\mathcal{F}\right)$. Also, the stalk at each point $x\in X$ is denoted by $\mathcal{F}_x$.\\ \\
\textbf{Acknowledgement:} I would like to thank Chitrabhanu Chaudhuri for several fruitful discussions. Discussions with Shulim Kaliman over email were very productive. I would also like to thank Ritwik Mukherjee for several suggestions. \\
\section{Reduction to normal case}
\begin{lmm}\label{birational_lemma}
Let $\phi: X \longrightarrow X$ be an endomorphism of an algebraic variety $X$ over $k$ and, let $Y$ be a proper closed subvariety of $X$ such that the restriction of $\phi$ to $X \setminus Y$ is injective. Then $\phi$ is birational.
\end{lmm}
\begin{proof}
Note that the morphism $\phi$ restricted to $X \setminus Y$, $$\widetilde{\phi} : X \setminus Y \longrightarrow X$$ is injective and $X \setminus Y$ is Z-open in $X$ together imply that $\widetilde{\phi}$ is a dominant morphism. Hence, the induced field extension $k(X ) \rightarrow k \left( X \setminus Y\right)$ between the function fields is algebraic. If $\left[ k(X):k(X \setminus Y)  \right]=n$, then there exists a non-empty Z-open set $U$ of $X$ such that the inverse image of every point of $U$ is made up of exactly $n$ points of $X \setminus Y$. Since $\widetilde{\phi}$ is injective, this shows that $n=1$. Hence, we conclude that $\widetilde{\phi}$ is birational. This shows that $\phi: X \longrightarrow X$ is also birational.
\end{proof}
\begin{lmm}\label{normal_case_reducing}
\Cref{main_conj} is true provided it is true for normal algebraic varieties.
\end{lmm}
\begin{proof}
Assume $\nu:\widetilde{X} \longrightarrow X$ is a normalization of $X$. The endomorphism $\phi$ as in \Cref{main_conj} induces an endomorphism $\rho:\widetilde{X} \longrightarrow \widetilde{X}$. By \Cref{birational_lemma}, $\phi$ is birational, so is $\rho$. In any case, $\rho$ restricted to $\widetilde{X} \setminus \nu^{-1}(Y)$ is quasi-finite. Hence, by the Zariski Main Theorem \cite[Section 9, chapter III]{Red_book}, $\rho$ restricted to $\widetilde{X} \setminus \nu^{-1}(Y)$ is injective, and hence $\rho:\widetilde{X} \longrightarrow \widetilde{X}$ is an automorphism by the hypothesis. We get the following commutative diagram:
\begin{displaymath}
\xymatrix{
\widetilde{X} \ar[r] ^{\rho}_{\simeq} \ar[d]_{\nu} & \widetilde{X} \ar[d]^{\nu}\\
X \ar[r]_{\phi} & X
}
\end{displaymath}
 The rest of the proof is essentially due to Kaliman \cite[Proof of Lemma 2]{Kaliman}. We prove it for the sake of completeness. Let $n(x)$ be the number of points of the fibre $\nu^{-1}(x)$ for $x \in X$, $k= \max \limits_{x \in X}n(x)$, and $X_k$ be the subvariety of $X$ consists of all points $x \in X$ such that $n(x)=k$. As $\rho(\nu^{-1}(x)) \subseteq \nu^{-1}(\phi(x))$ for all $x \in X$, and $\rho$ is bijective, $n(\phi(x)) \geq n(x)$. Hence $\phi(x) \in X_k$ for each $x \in X_k$. Hence, $\phi$ restricted to $X_k$ is an endomorphism. By the same reasoning, we conclude that $\phi:X_k \longrightarrow X_k$ is injective, and hence surjective by the theorem of Ax. On the other hand, $\phi:X \setminus X_k \longrightarrow X\setminus X_k$ is an endomorphism and we replace $X$ by $X \setminus X_k$ in the above argument to conclude that $\phi(X_{k-1}) \subseteq X_{k-1}$ and the restriction map $\phi: X_{k-1} \longrightarrow X_{k-1}$ is injective. Inductively we conclude that $\phi:X \longrightarrow X$ is injective, and hence an automorphism by the theorem of Ax.    
\end{proof}

\section{Proof of the Theorems}
\hf Without loss of generality, we assume $X$ is normal from now onwards. \Cref{main_thm_quasi_finite_case} is now easy to prove.\\
 \begin{myproof}{\Cref{main_thm_quasi_finite_case}}
 Note that $\phi$ is birational, quasi-finite and $X$ is normal; by the Zariski Main Theorem, $\phi$ is injective, and hence an automorphism by the theorem of Ax.
 \end{myproof}
 \begin{rem}
 \Cref{main_thm_quasi_finite_case} can be extended from varieties to those schemes for which the theorem of Ax hold. 
 \end{rem}
 \hf From now onwards, we focus on the case where $\text{codim}_XY \geq 2$ and $k=\C$. The following lemma, due to Kaliman, will be useful for our purpose.
\begin{lmm}
Let $\phi$ be an endomorphism of an algebraic variety $X$ over $k$. Also assume that $Y$ be a closed subvariety of $X$ of codimension at least $2$ such that the restriction map $\phi:X \setminus Y \longrightarrow X$ is injective. Let $Z= X \setminus \phi (X \setminus Y)$. Then $Z$ is a closed subvariety of $X$ and $\dim Y= \dim Z$.
\end{lmm}
\begin{proof}
\cite[Lemma 3]{Kaliman}.
\end{proof}

\hf  Note that the map $\phi: X \setminus Y \rightarrow X$ is a injective birational map and hence, by the Zariski Main Theorem, it is an embedding. Therefore $\phi (X \setminus Y)$ is Z-open in $X$ and $Z$ is Z-closed. Hence, we get that the morphism 
\begin{align}\label{psi_map}
\phi:X \longrightarrow X
\end{align}
with the property that the restriction morphism $\phi:X \setminus Y \longrightarrow X \setminus Z$ is an isomorphism.  Let us denote the analytic spaces $X^{\text{an}} \setminus Y^{\text{an}}$ and $X^{\text{an}} \setminus Z^{\text{an}}$ by $(U,\cH_U)$ and $(V,\cH_V)$ respectively. Also, denote the inclusion of $U \hookrightarrow X^{\text{an}} $ and $V \hookrightarrow X^{\text{an}}$ by $i_U$ and $i_V$ respectively. Hence, we will get the following diagram of morphism of locally ringed spaces:

\begin{equation}\label{com_diag_1}
\begin{gathered}
\xymatrix{
(U,\cH_U) \ar[r]^{\phi} \ar[d]_{i_U} & (V,\cH_V) \ar[d]^{i_V}\\
(X^{\text{an}},\cH_X) \ar[r]_{\phi} & (X^{\text{an}},\cH_X)
}
\end{gathered}
\end{equation}
On the level of structures, we will get the following commutative diagram of morphism of sheaf of rings over $X^{\text{an}}$:
\begin{equation}\label{struc_com_diag_1}
\begin{gathered}
\xymatrix{
\cH_X \ar[r] \ar[d] & \phi_{*}\cH_X\ar[d] \\
i_{V_*}\cH_V \ar[r] & i_{V_*}\phi_{*}\cH_U
}
\end{gathered}
\end{equation}
\hf  Observe that, both $U$ and $V$ are open subsets of $X^{\text{an}}$ with the property that $X^{\text{an}} \setminus U$ and $X^{\text{an}} \setminus U$ both are closed subvarieties of $X^{\text{an}}$ of codimension at least $2$. The variety $X$ being normal in Zariski topology is irreducible. Therefore, both $U$ and $V$ are open and dense in Zariski topology of $X$, so they are dense in the analytic topology of $X$ as well \cite[Proposition 5]{GAGA}. Therefore, both $U$ and $V$ always have non-empty intersection individually with any non-empty open subset of $X^{\text{an}} $. Again, $X^{\text{an}} $ is normal follows from normality of the algebraic variety $X$. Let $W$ be a non-empty open subset of $X^{\text{an}}$; then $W$ is normal as well. Also, both $W \cap U \neq \emptyset$ and $W \cap V \neq \emptyset$; both $W \cap Y^{\text{an}} $ and $W \cap Z^{\text{an}}$ are closed analytic subset of $W$ of codimension at least $2$. This is because analytically open subsets of an algebraic variety $X$ is analytic of dimension $\dim X$; in particular, $\dim W \cap Y^{\text{an}}=\dim Y^{\text{an}}$, $\dim W \cap Z^{\text{an}}=\dim Z^{\text{an}}$ and $\dim W=\dim X^{\text{an}}$. According to Riemann's 2nd removable singularity theorem \cite[Proposition 4, Chapter VI]{Narasimhan}, both the natural ring maps $\Gamma \left( W, \cH_X\right) \rightarrow  \Gamma \left( W \cap U, \cH_X\right)$ and $\Gamma \left( W, \cH_X\right) \rightarrow  \Gamma \left( W \cap V, \cH_X\right)$ are isomorphism. Hence, we get both the natural morphisms of sheaf of rings $\cH_X \rightarrow i_{U_*} \cH_U$ and $\cH_X \rightarrow i_{V_*} \cH_V$ over $X^{\text{an}}$ are isomorphism. On the other hand, the isomorphism $\phi:(U,\cH_U)\rightarrow (V,\cH_V)$ produces the natural morphism of sheaf of rings $ \cH_V  \rightarrow  \phi_* \cH_U$ over $X^{\text{an}} $ is an isomorphism. Therefore, we get all the arrow in the diagram \eqref{struc_com_diag_1} except the horizontal one on the top are isomorphism. Hence the morphism of sheaf of rings over  $X^{\text{an}} $
\begin{align} \label{push_forward_isom}
 \cH_X  \longrightarrow \phi_{*}\cH_X
\end{align}
 is an isomorphism. \\
 \hf The next proposition plays a central role to prove \Cref{main_thm_codim_case}.
 \begin{prp} \label{intersection_commutes_with_image_lemma}
Consider $X, \phi$ as in \cref{psi_map}. Then $\phi(A \cap U)=\phi(A) \cap V$ for any open set $A$ in $X^{ \text{an}}$.
 \end{prp}
 \begin{proof}
 Let $A$ be an open subset of $X^{ \text{an}}$. Also assume $p\in A$; we will first show that $\phi(p) \in \overline{\phi(A \cap U)}$. If this is not the case, then there exist open set $B$ of $ X^{ \text{an}}$ such that $\phi(p) \in B$ and $B \cap \phi(A \cap U) = \emptyset $. But $p \in \phi^{-1}(B) \cap A$, therefore $ \phi^{-1}(B) \cap (A \cap U) \neq \emptyset$ by the density property of $U$, i.e. $ \phi^{-1}(B) \cap \phi^{-1}\phi(A \cap U) \neq \emptyset$, which is a contradiction. Therefore we conclude that $\phi(p) \in \overline{\phi(A \cap U)}$ for every open set $A$ of  $X^{ \text{an}}$ containing $p$. \\
\hf Next, we will show that $\phi(A) \cap V \subseteq \phi(A \cap U)$. If $p \in A$ with $\phi(p) \in \phi(A) \cap V$, choose an open set $B$ of $X^{ \text{an}}$ containing $p$ such that $\overline{B} \subseteq A$. This can be possible because $X^{ \text{an}}$ is Hausdorff and locally compact space. Hence, $\overline{B \cap U}\cap U \subseteq A \cap U$. As $\phi$ restricted to $U$ is a homeomorphism, $\overline{ \phi(B \cap U)} \cap V \subseteq \phi(A \cap U)$. Since $\phi(p) \in \overline{\phi(B \cap U)}$ according to the discussion above and we assumed $\phi(p)\in V$, so $\phi(p) \in \phi(A \cap U)$. This proves the fact $\phi(A) \cap V \subseteq \phi(A \cap U)$. Again, $\phi(A \cap U) \subseteq \phi(A)\cap V$ is obvious. Therefore $\phi(A \cap U)=\phi(A) \cap V$.
 \end{proof}
 \hf The following corollary is an immediate consequence of \Cref{intersection_commutes_with_image_lemma}.
 \begin{cor}\label{inverse_image_coincide_cor}
 Consider $X, \phi$ as in \cref{psi_map}. Then $\phi^{-1}(V)=U$ where both $U$ and $V$ are defined as above.
 \end{cor}
 \begin{proof}
 We will first show that for each point $p$ of $\phi^{-1}(V)$, the analytic stalk map $$\phi^{\text{an}}_{p}: \cH_{X,\phi(p)} \longrightarrow \cH_{X,p}$$ is surjective. Indeed, if $A$ is an open subset of $X^{ \text{an}}$, then it follows from the isomorphism $\cH_X \simeq i_{U_*}\cH_U$ that $\Gamma(A, \cH_X) \simeq \Gamma(A\cap U, \cH_X)$ and from $\phi^{-1}\cH_V \simeq \cH_U$ that $\Gamma(A\cap U, \cH_X) \simeq \Gamma(\phi\left(A\cap U\right), \cH_X) \simeq \Gamma \left( \phi(A) \cap V,\cH_X \right)$ by \Cref{intersection_commutes_with_image_lemma}. Observe that $\phi(A) \cap V$ is an open subset of  $X^{ \text{an}}$ containing the $\phi(p)$. The surjectivity of $\phi^{\text{an}}_{p}$ follows.\\
\hf Now apply  \cite[Proposition 1, Chapter IV]{Narasimhan} to conclude that  $p$ is an isolated point of  $\phi^{-1}(\phi(p))$. Therefore for each point $q \in V$, the fiber $\phi^{-1}(q)$ is a discrete set. Now considering $\phi^{-1}(q)$ as an algebraic fibre, we conclude that it is finite. Therefore the restriction morphism $\phi: \phi^{-1}(V) \longrightarrow V$ is quasi-finite. \\
\hf Note that we can assume $V$ as an algebraic variety over $\C$ and hence, $\phi$ restricted to $\phi^{-1}(V)$ as an algebraic morphism. Also note that the restriction morphism is birational. By the Zariski Main Theorem,  $\phi: \phi^{-1}(V) \longrightarrow V$ is injective. Again $\phi: U \longrightarrow V$ is injective and $U \subseteq \phi^{-1}(V)$. It follows that $\phi^{-1}(V)=U$.
\end{proof} 
\begin{lmm}\label{closure_lemma}
Let $X$ be an algebraic variety and $Y$ be a closed subvariety of $X$. If $A$ be a non-empty open subset of $X^{ \text{an}}\setminus Y^{ \text{an}}$, then $\overline{A}$, the closure of $A$ in the analytic topology of $X^{ \text{an}}$, is analytic of pure dimension $\dim X$. 
\end{lmm}
\begin{proof}
Note that $X^{ \text{an}}\setminus Y^{ \text{an}}$ is open in $X^{ \text{an}}$; so $A$ is open in $X^{ \text{an}}$ as well. Therefore $A$ is an analytic subset of $X^{ \text{an}}$. By \cite[Proposition $4^{\prime}$, Chapter IV]{Narasimhan}, $\overline{A}$ is analytic.\\
\hf Let $B$ denote the set of all regular points of the algebraic variety $X$. Then $B$ is Z-open and Z-dense in $X$, and hence, by \cite[Proposition 5]{GAGA}, $B$ is dense in $X^{ \text{an}}$ as well. Therefore $A \cap B$ is a non-empty dense open subset of $A$. Hence $\overline{A \cap B}=\overline{A}$. Because of the density of $A \cap B$ in $A$, all the irreducible component of $\overline{A}$ must intersect with $A \cap B$. So, $\overline{A}$ is of pure dimension $\dim A \cap B$. It is easy to check that $\dim A \cap B= \dim X$.
\end{proof}
\hf The next proposition plays a central role to prove \Cref{main_thm_codim_case}.
\begin{prp}\label{analytic_stalk_integral_proposition}
 Consider $X, \phi$ as in \cref{psi_map}. Then for each point $p \in X$, the corresponding morphism of analytic stalk  
\begin{equation*}
\phi^{\text{an}}_{p}:\cH_{X,\phi(p)} \longrightarrow \cH_{X,p}
\end{equation*}
 is an integral extension.
\end{prp}
\begin{proof}
We will first show that $\phi^{\text{an}}_{p}$ is injective. Let $A$ be an open subset of $X^{ \text{an}}$ containing $\phi(p)$ and $f$ be holomorphic function which defines a germ at $\phi(p)$ such that $f \circ \phi \equiv 0$ for some open set $B\subseteq \phi^{-1}(A)$ containing $p$. 
%Since $\left( B,\cH_X|_B\right)$ is normal, it is locally irreducible; we then choose an irreducible open analytic subset $C$ of $B$ containing $\phi(p)$.
Therefore $f \circ \phi \equiv 0$ on $B \cap U$. As $\phi:U \longrightarrow V$ is a homeomorphism, using \Cref{intersection_commutes_with_image_lemma}, we conclude that $f \equiv 0$ on $\phi(B) \cap V$. Since $\phi(B) \cap V$ is dense in $\phi(B)$, i.e. $\overline{\phi(B) \cap V}= \overline{\phi(B)}$,  apply \Cref{closure_lemma} to conclude that $f \equiv 0$ on the analytic space $\overline{\phi(B)}\cap A$. Therefore by \cite[Definition 4, Chapter III]{Narasimhan}, we conclude that there exist an open subset $D$ of $X^{ \text{an}}$ containing $\overline{\phi(B)}\cap A$ such that $f \equiv 0$ on $ D$. Since $p \in D$, $\phi^{\text{an}}_{p}$ is injective.\\
\hf  Let $C$ be an open subset of $X^{ \text{an}}$ containing $p$ and $f \in \Gamma \left( C,\cH_X \right)$ which defines a germ of holomorphic function at $p$. Without loss of generality we may assume that $C$ is irreducible open analytic subset of $X^{ \text{an}}$ as $X^{ \text{an}}$ being normal, is locally irreducible. Now, $f|_{U}$, the restriction of $f$ to $U$, is an element of $\Gamma \left(C \cap U,\cH_X \right)$. From the isomorphism $\phi^{-1}\cH_V \simeq \cH_U$, it follows that $f|_{U}$ induces an unique holomorphic map $g \in \Gamma \left( \phi(C \cap U),\cH_X \right)$, i.e. $g$ is a holomorphic function defined on $\phi(C) \cap V$ by \Cref{intersection_commutes_with_image_lemma}. \\
\hf  Now we will show that $g$ can be extended to a weakly holomorphic function on $\overline{\phi(C)}$. Note that $\overline{\phi(C)}$ is analytic subset of $X^{ \text{an}}$ of pure dimension $\dim X$ by \Cref{closure_lemma}. Also note that the analytic space $Z^{\text{an}}$ was defined to be the complement of $V$ in $X^{\text{an}}$ and its codimension is at least $2$. Therefore $Z^{\text{an}}\cap \overline{\phi(C)}$ has codimension at least $2$ in $\overline{\phi(C)}$. According to \cite[Proposition 12, Chapter III]{Narasimhan}, $g$ can be extended to a unique holomorphic function $h$ on $\overline{\phi(C)}_{\text{reg}}$, the set of all regular points of the analytic space $\overline{\phi(C)}$.  We may assume $h$ is holomorphic on $\left(\phi(C) \cap V\right)\cup \overline{\phi(C)}_{\text{reg}}$. The analytic space of all singular points of $\overline{\phi(C)}$ which lies in $Z^{\text{an}}\cap \overline{\phi(C)}$ is of codimension at least $2$. According to \cite[Proposition 6.1, Chapter 2]{Demailly}, $h$ is a weakly holomorphic function on $\overline{\phi(C)}$.\\
\hf We denote the germ of weakly holomorphic functions at $\phi(p)\in \overline{\phi(C)}$ by $\widetilde{\cH}_{\overline{\phi(C)},\phi(p)}$ in this sequel. Then $h$ as an element of $\widetilde{\cH}_{\overline{\phi(C)},\phi(p)}$, is integral over $\cH_{\overline{\phi(C)},\phi(p)}$, the germ of holomorphic functions at $\phi(p)\in \overline{\phi(C)}$. On the other hand, note that the morphism $\cH_{X,\phi(p)} \longrightarrow \cH_{\overline{\phi(C)},\phi(p)}$ is surjective, and hence, $h$ satisfies a monic polynomial $\Phi[z]$ with the coefficients in $\cH_{X,\phi(p)}$. Therefore, there exist an open subset $D$ of $X^{ \text{an}}$ such that $\Phi(h)\equiv 0$ on $D$. Also note that $D \cap (\phi(C)\cap V) \neq \emptyset$ as $\phi(C)\cap V$ is dense in $\overline{\phi(C)}$.
\\
\hf On the other hand, note that $h \circ \phi = f$ on $C \cap U$. Hence, $\Phi(f) \equiv 0$ on some non-empty open subset of $C$, namely $\phi^{-1}(D)\cap \left(C \cap U\right)$. Note that the set $\phi^{-1}(D)\cap \left(C \cap U\right)$ is non-empty because $\phi$ restricted to $U$ is bijective to $V$ and $D \cap (\phi(C)\cap V) \neq \emptyset$. Since we assumed $C$ to be an irreduciable analytic space, by \cite[Theorem 6.2, Chapter 2]{Demailly}, we conclude that $\Phi(f)$ is identically zero on $C$. Hence, we can say that the germ defined by the holomorphic function $f$ at $p$ satisfies a monic polynomial with the coefficients in $\cH_{X,\phi(p)}$  via $\phi^{\text{an}}_{p}$. Therefore it is integral over $\cH_{X,\phi(p)}$. This completes the proof.  
\end{proof}

\hf The following corollary plays a crucial role to prove \Cref{main_thm_codim_case}.
\begin{cor}\label{algebraic_stalk_isom_cor}
Let $\phi$ be an endomorphism of a complex algebraic variety $X$ and, let $Y$ be a proper closed subvariety of $X$ such that the restriction of $\phi$ to $X \setminus Y$ is injective. If $\text{codim}_YX \geq 2$, then for each point $p \in X$, the corresponding morphism of algebraic stalk  
\begin{equation*}
\phi_{p}:\cO_{X,\phi(p)} \longrightarrow \cO_{X,p}
\end{equation*}
is an isomorphism.
\end{cor}
\begin{proof}
Consider the following commutative diagram of stalks 
\begin{equation}\label{stalk_diagram}
\begin{gathered}
\xymatrix{
\cO_{X,\phi(p)} \ar[r]^{\phi_{p}} \ar[d]_{\alpha} & \cO_{X,p} \ar[d]^{\beta}\\
\cH_{X,\phi(p)} \ar[r]_{\phi^{\text{an}}_{p}} &  \cH_{X,p}
}
\end{gathered}
\end{equation}
where all but $\phi_{p}$ are injective. The arguments of injectivity of $\phi^{\text{an}}_{p}$ is given in the proof of \Cref{analytic_stalk_integral_proposition}. So, $\phi_{p}$ is injective. Throughout this proof, we will use the notations introduced in the diagram above for the maps respectively. We need to prove $\phi_{p}$ is surjective. Note that all the rings in the diagram are normal domain. Also $X$ is normal and $\phi$ is birational; the field of fractions of $\cO_{X,p}$ and $\cO_{X,\phi(p)}$ are the function field $k(X)$ of $X$ and the natural induced map between the field of fractions is an isomorphism. If we denote the field of fractions of $\cH_{X,p}$ and $\cH_{X,\phi(p)}$ by $\mathcal{M}_{p}$ and $\mathcal{M}_{\phi(p)}$ respectively, then the field of fraction of each ring will produce the following commutative diagram
\begin{displaymath}
\xymatrix{
k(X) \ar[r]_{\simeq}^{\widetilde{\phi}_p} \ar[d]_{\widetilde{\alpha}} & k(X) \ar[d]^{\widetilde{\beta}}\\
\mathcal{M}_{\phi(p)} \ar[r]_{\widetilde{\phi}^{\text{an}}_{p}} & \mathcal{M}_{p}
}
\end{displaymath}
Now let $f \in \cO_{X,p}$. We can assume $f \in k(X)$, the field of fraction of $\cO_{X,p}$; there exist an unique element $g$ of $k(X)$, the field of fraction of $\cO_{X,\phi(p)}$ such that $\widetilde{\phi}_p(g)=f$. We will map abuse of notation by denoting $\widetilde{\alpha}(g)$ by $g$. Note that $\widetilde{\phi}^{\text{an}}_{p}(g)=\beta(f)$, where $\beta(f)$ is thought as an element of $\mathcal{M}_{p}$. Since $\beta(f)$ is integral over $\cH_{X,\phi(p)}$ by \Cref{analytic_stalk_integral_proposition},  $g $ is integral over $\cH_{X,\phi(p)}$ as well, and hence, $g \in \cH_{X,\phi(p)}$. Finally, we get $g \in k(X) \cap \mathcal{H}_{X,\phi(p)}$ in $\mathcal{M}_{\phi(p)}$. If we write $g=\dfrac{a}{b}$ for some $a,b \in \cO_{X,\phi(p)}$, then $a=b \cdot g$ in $\mathcal{H}_{X,\phi(p)}$; and hence, $a=b \cdot g$ in the completion ring of $\mathcal{H}_{X,\phi(p)}$. As the completions of both $\mathcal{H}_{X,\phi(p)}$ and $\mathcal{O}_{X,\phi(p)}$ with respct to their maximal ideals respectively, are isomorphic \cite{GAGA}, $b|a$ \footnote{$b |a$ means $b$ divides $a$} in the completion ring of $\cO_{X,\phi(p)}$. Now by \cite[Lemma 1.29, Chapter 1]{Mumford}, $b|a$ in $\cO_{X,\phi(p)}$. Hence $g \in \cO_{X,\phi(p)}$. This completes the proof.
\end{proof}
\hf Now it is easy to prove \Cref{main_thm_codim_case}.\\
\begin{myproof}{\Cref{main_thm_codim_case}}
It is immediate from \Cref{algebraic_stalk_isom_cor} that $\phi:X \longrightarrow X$ is flat. According to \Cref{inverse_image_coincide_cor}, all points of $V$ has zero-dimensional fibre, $\dim \phi^{-1}(y)=0$ for every $y \in X$ as the fibers are of constant dimension under flat morphism. Therefore, $\phi$ is quasi-finite. Since $X$ is normal and $\phi$ is birational, by the Zariski Main Theorem, $\phi$ is injective and hence, an automorphism by the theorem of Ax.
\end{myproof}

\bibliographystyle{alpha}
\bibliography{ref}

\end{document}